\documentclass{amsart}

\usepackage{amsmath, amssymb, amscd, amsthm}
\usepackage{geometry}
\usepackage{needspace}

\newtheorem{thm}{Theorem}[section]
\newtheorem{lem}[thm]{Lemma}
\newtheorem{prop}[thm]{Proposition}
\newtheorem{cor}[thm]{Corollary}
\theoremstyle{definition}
\newtheorem{counterexample}[thm]{Counterexample}
\numberwithin{equation}{section}  

\newcommand{\set}[1]{\left\{ #1  \right\}}
\newcommand{\MK}[1]{M_\K({#1})}
\newcommand{\MX}[1]{M_\X({#1})}
\newcommand{\MY}[1]{M_\Y({#1})}
\newcommand{\PX}[1]{P_\X({#1})}
\newcommand{\PY}[1]{P_\Y({#1})}
\newcommand{\PK}[1]{P_\K({#1})}
\newcommand{\Phii}{\Phi^{-1}}
\newcommand{\xo}{x_1}
\newcommand{\xz}{x_0}
\newcommand{\yz}{y_0}
\newcommand{\yo}{y_1}
\newcommand{\X}{X}
\newcommand{\Y}{Y}
\newcommand{\K}{K}
\newcommand{\SCX}{{S_{\X}^+}}
\newcommand{\SCY}{{S_{\Y}^+}}
\newcommand{\SCK}{{S_{\K}^+}}
\newcommand{\BK}{B_{\K}^{+}}
\newcommand{\BX}{B_{\X}^{+}}
\newcommand{\BY}{B_{\Y}^{+}}
\newcommand{\GK}[1]{{G_K}{(#1)}}
\newcommand{\GX}[1]{{G_X}{(#1)}}
\newcommand{\GY}[1]{{G_Y}{(#1)}}
\newcommand{\B}{B_{I}^{+}}
\newcommand{\Psir}{\Psi_r}
\newcommand{\G}{\mathcal{G}_{p}}
\newcommand{\Sp}{\mathcal{S}_{p}^+}
\newcommand{\Sinf}{{\mathcal{S}_{\infty}^+}}
\newcommand{\SCL}{S_{L}^+}
\newcommand{\V}[1]{\Vert{#1}\Vert}
\newcommand{\Vinf}[1]{\Vert{#1}\Vert_{(\infty)}} 
\newcommand{\Vp}[1]{\Vert{#1}\Vert_{(p)}}
\newcommand{\VI}[1]{\Vert{#1}\Vert_{I}}
\newcommand{\VK}[1]{\Vert{#1}\Vert_{K}}
\newcommand{\VL}[1]{\Vert{#1}\Vert_{L}}
\newcommand{\VX}[1]{\Vert{#1}\Vert_{X}}
\newcommand{\VY}[1]{\Vert{#1}\Vert_{Y}}
\newcommand{\Jp}{J_p}
\newcommand{\Jinf}{J_\infty}

\title[Norm-additive maps on $\mathcal{S}_{p}^+$]{Norm-Additive Maps on the Positive Unit Sphere of $C^1([0,1])$}

\author[K.~Ezumi]{Kazuki Ezumi}
\thanks{Kazuki Ezumi:
Graduate School of Integrated Arts and Sciences,
Niigata University,
Niigata 950-2181, Japan.
E-mail (corresponding author): y26e023k@mail.cc.niigata-u.ac.jp}

\author[T.~Miura]{Takeshi Miura}
\thanks{Takeshi Miura:
Department of Mathematics,
Faculty of Science,
Niigata University,
Niigata 950-2181, Japan.
E-mail: miura@math.sc.niigata-u.ac.jp}
\dedicatory{Dedicated to Professor Ngai-Ching Wong on the occasion of his 65th birthday}
\begin{document}


\begin{abstract}
Let $C^1(I)$ denote the space of all continuously differentiable
real-valued functions on the closed unit interval $I=[0,1]$.
For $1\le p\le\infty$, we equip $C^1(I)$ with the norm
\[
\Vp{f}
:=
\begin{cases}
\left(|f(0)|^p+\VI{f'}^p\right)^{1/p},
&
1\le p<\infty,
\\
\max\{|f(0)|,\VI{f'}\},
&
p=\infty,
\end{cases}
\qquad
(f\in C^1(I)),
\]
where $\VI{\cdot}$ denotes the supremum
norm on $I$.
Let 
$\Sp:=\set{f\in C^1(I):f(0)\ge0,\,f'\ge0,\,\Vp{f}=1}$
be the positive unit sphere of $C^1(I)$.
A map $T:\mathcal{S}_p^+\to \mathcal{S}_p^+$ is called \emph{norm-additive} if
\[
\|T(f)+T(g)\|_{(p)}=\|f+g\|_{(p)}
\qquad (f,g\in \mathcal{S}_p^+).
\]
For $1<p\le\infty$, we prove that every surjective norm-additive map
$T\colon\Sp\to\Sp$ is of the form
\[
T(f)(t)
=
f(0)+\int_0^t f'(\tau(s))\,ds
\]
for all $f\in\Sp$ and $t\in I$, where
$\tau\colon I\to I$ is a homeomorphism.
Conversely, every map of this form is a surjective norm-additive map.
Consequently, every such map extends to a surjective linear isometry
from $(C^1(I),\Vp{\cdot})$ onto itself.
For $p=1$, however, we construct a surjective norm-additive map on
$\mathcal{S}_1^+$ which
is not of the above form and admits no linear extension to
$C^1(I)$,
showing that the assumption $p>1$ is essential.
\end{abstract}

\maketitle

\section{Introduction}
Extension problems for mappings defined on unit spheres have been studied
extensively in Banach space theory. A classical problem in this direction is
Tingley's problem. In 1987, Tingley~\cite{ting} asked whether every
surjective isometry between the unit spheres of two Banach spaces extends to
a surjective linear isometry between the whole spaces. Since then, Tingley's
problem and its variants have been investigated for many classes of Banach
spaces; see, for example,
\cite{CHMP,CP,Hatori2022,Hatori2023,HOS,HM,peralta,Wang1994,Wang1996}
and the references therein. Variants in which the mapping is defined only on
the positive unit sphere have also attracted considerable attention.

Let $K$ be a compact Hausdorff space, and let $C(K)$ denote the Banach space
of all real-valued continuous functions on $K$, equipped with the supremum
norm
\[
\VK{f}:=\sup_{x\in K}|f(x)|
\qquad (f\in C(K)).
\]
We write
\[
\SCK:=\set{f\in C(K):f\geq0,\ \VK{f}=1}
\qquad\text{and}\qquad
\BK:=\set{f\in C(K):f\geq0,\ \VK{f}\leq1},
\]
and call them the positive unit sphere and the positive unit ball of $C(K)$, respectively.
Leung, Ng, and Wong~\cite{LNW} studied surjective isometries
between the positive unit spheres of $C(X)$ and $C(Y)$, where $X$ and $Y$
are compact Hausdorff spaces, and showed that every such isometry is induced
by a homeomorphism between $X$ and $Y$.
In particular, it extends to a
surjective linear isometry between $C(X)$ and $C(Y)$.

Another class of mappings related to such rigidity and extension phenomena is
given by norm-additive maps. Let $(E,\|\cdot\|_E)$ and $(F,\|\cdot\|_F)$ be
normed spaces, and let $E_0 \subset E$ and $F_0 \subset F$. A map
$T \colon E_0 \to F_0$ is called \emph{norm-additive} if
\[
\|T(a)+T(b)\|_F
=\|a+b\|_E
\qquad (a,b\in E_0).
\]

Norm-additive maps on positive subsets of ordered Banach spaces have been
studied in several settings. 
For an ordered normed space $E$, we denote its positive cone by
$E^+:=\set{a\in E:a\geq0}$.
Zhang, Tsai, and Wong~\cite{ZTW_lattice} formulated a general
extension problem for such maps: whether every surjective norm-additive map
from $E^+$ onto $F^+$, where $E$ and $F$ are ordered Banach spaces,
extends to a linear map from $E$ onto $F$.
For $L^p$-spaces with $1<p\leq\infty$, they obtained linear isometric
extension results, whereas the case $p=1$ admits a counterexample.
They subsequently studied the corresponding problem for noncommutative
$L^p$-spaces in~\cite{ZTW_noncommutative}.
Moving from positive cones to positive unit spheres,
Hao, Dong, and Li~\cite{HDL} studied surjective norm-additive maps
between the positive unit spheres of $L^p$-spaces for $1<p<\infty$
and obtained a surjective isometric extension result.

In a different direction,
Chen, Dong, and Zheng~\cite{CDZ} established a
Banach--Stone type representation theorem for surjective norm-additive maps
between the maximal groups of positive continuous functions.
More recently,
Sun and Sun~\cite{SS} studied the extension problem for bijective
norm-additive maps on the positive unit spheres of $\ell_q(\ell_p)$ 
for
$1\leq p,q\leq\infty$.
They obtained linear isometric extension results in some cases and
counterexamples in certain endpoint cases.
These results indicate that norm-additive extension phenomena depend
sensitively on the geometry of the underlying positive set and that endpoint
exponents can exhibit genuinely different behavior.

We now turn to the function space considered in this paper. Let
$C^1(I)$ denote the space of all continuously differentiable real-valued
functions on the closed unit interval $I=[0,1]$. For
$1\leq p\leq\infty$, we equip $C^1(I)$ with the norm
\[
\Vp{f}
:=
\begin{cases}
\left(|f(0)|^p+\VI{f'}^p\right)^{1/p},
& 1\leq p<\infty,\\[1mm]
\max\{|f(0)|,\VI{f'}\},
& p=\infty,
\end{cases}
\qquad
(f\in C^1(I)).
\]
For $1\leq p\leq\infty$, let
\[
\Sp:=\set{f\in C^1(I):f(0)\geq0,\ f'\geq0,\ \Vp{f}=1}
\]
denote the corresponding positive unit sphere.
Of particular relevance to the present paper is a recent preprint by
Enami, Hirota, Koshimizu, and Lin~\cite{enam}.
They showed that, for every
$1\leq p\leq\infty$, each surjective isometry on $\Sp$ admits a surjective
linear isometric extension to $C^1(I)$. 
Motivated by this result and by
the extension theory of norm-additive maps described above, we investigate
surjective norm-additive maps
$T\colon\Sp\to\Sp$.
\Needspace{14\baselineskip}
Our main results are as follows.
\begin{thm}\label{thm1.1}
Let $1<p\leq\infty$.
Let $T\colon\Sp\to\Sp$ be a surjective norm-additive map.
Then there exists a homeomorphism $\tau\colon I\to I$ such that
\[
T(f)(t)
=
f(0)+\int_0^t f'(\tau(s))\,ds
\]
for all $f\in\Sp$ and $t\in I$.
Conversely, every map given by the above formula is a
surjective norm-additive map.
\end{thm}

As a consequence, we obtain the following linear extension result.

\begin{cor}\label{cor1.2}
Let $1<p\leq\infty$.
Let $T\colon\Sp\to\Sp$ be a surjective norm-additive map.
Then $T$ extends to a surjective linear isometry from
$(C^1(I),\Vp{\cdot})$ onto
itself.
\end{cor}

The forward implication of Theorem~\ref{thm1.1} and the conclusion of
Corollary~\ref{cor1.2} both fail for $p=1$.
Indeed, we construct a surjective norm-additive map on
$\mathcal{S}_1^+$ which admits no linear extension to $C^1(I)$,
and hence cannot have the representation in Theorem~\ref{thm1.1}.

We briefly describe the main ideas of the proof. 
For $1<p<\infty$, the derivative map $D_p(f)=f'$ gives a bijection from
$\Sp$ onto $\B$. 
A key step is to prove that a surjective norm-additive map preserves the initial value $f(0)$. This makes it possible to transfer $T$ to a surjective norm-additive map on $\B$, to which an auxiliary Banach--Stone
type representation theorem for positive unit balls can be applied.
For $p=\infty$, we set $L:=\{-1\}\cup I$ and encode $f(0)$ and $f'$ as a single continuous function on $L$, thereby
identifying $\Sinf$ with $\SCL$.
An auxiliary Banach--Stone type representation theorem for positive unit
spheres can then be applied.

The remainder of this paper is organized as follows.
In
Section~\ref{section2}, we establish the auxiliary representation theorems
for surjective norm-additive maps on positive unit
spheres and positive unit balls of $C(K)$.
Since related representation results for norm-additive maps on positive
continuous functions are already known, notably~\cite{CDZ}, these theorems
are used here only as auxiliary tools for proving our main results on $\Sp$.
In Section~\ref{section3}, we prove the forward implication
of Theorem~\ref{thm1.1} for $1<p<\infty$. In Section~\ref{section4}, we
treat the case $p=\infty$. In Section~\ref{section5}, we prove the converse implication
and establish Corollary~\ref{cor1.2}. Finally, in Section~\ref{section6}, we
construct the counterexample for $p=1$.

\section{Auxiliary representation theorems for positive spheres and balls of $C(K)$}\label{section2}

In this section, we establish two auxiliary representation theorems
for surjective norm-additive maps on positive unit spheres and
positive unit balls of $C(K)$.
We first treat the positive unit sphere case and then derive the
corresponding result for the positive unit ball by a normalization
argument.

We write
\[
\SCK:=\{f\in C(K):f\geq0,\,\V{f}_K=1\}
\]
for the positive unit sphere of $C(K)$.

\subsection{The positive unit sphere case}

We first establish the main result of this subsection, namely,
a representation theorem for the positive unit sphere.
Chen, Dong, and Zheng~\cite[Theorem 2.6]{CDZ} proved a similar representation theorem for surjective norm-additive maps between the strictly positive unit spheres of \(C(X)\) and \(C(Y)\). Since the positive unit spheres considered here also contain functions having zeros, their result does not directly cover the present setting. We therefore give a proof adapted to our setting.

\begin{thm}\label{Thm:auxiliary representation theorem}
Let $X$ and $Y$ be compact Hausdorff spaces.
Let $\Phi\colon\SCX\to\SCY$
be a surjective norm-additive map.
Then there exists a homeomorphism
$\tau\colon\Y\to\X$ such that
$\Phi(f)(y)=f(\tau(y))$
for all $f\in\SCX$ and $y\in Y$.
\end{thm}

Throughout the remainder of this subsection, let $X$ and $Y$ be
compact Hausdorff spaces, and let
$\Phi\colon\SCX\to\SCY$ be a surjective norm-additive map.
That is,
\[
\VY{\Phi(f)+\Phi(g)}=\VX{f+g}
\qquad (f,g\in\SCX).
\]

For $f\in\SCK$, set
\[
\GK{f}:=\{g\in\SCK:\VK{f+g}=2\},
\]
which we call the \emph{star set} of $f$ in $\SCK$.
We first record a basic lemma on surjective norm-additive maps.

\begin{lem}\label{lem:star-set-preservation}
For every $f\in\SCX$,
$\Phi(\GX{f})=\GY{\Phi(f)}$.
\end{lem}

\begin{proof}
Let $u\in\Phi(\GX{f})$.
There exists $g\in\GX{f}$ such that
$u=\Phi(g)$.
By the definition of $\GX{f}$,
we have $\VX{f+g}=2$.
Since $\Phi$ is norm-additive,
\[
\V{\Phi(f)+u}_Y
=\V{\Phi(f)+\Phi(g)}_Y
=\V{f+g}_X=2.
\]
Thus $u\in\GY{\Phi(f)}$.
Since $u\in\Phi(\GX{f})$ was arbitrary,
we conclude that
$\Phi(\GX{f})\subset\GY{\Phi(f)}$.

Conversely, let $v\in\GY{\Phi(f)}$.
By definition, $v\in\SCY$ with
$\V{\Phi(f)+v}_Y=2$.
Since $\Phi\colon\SCX\to\SCY$ is surjective,
there exists $h\in\SCX$ such that
$\Phi(h)=v$.
The norm-additivity of $\Phi$ gives
\[
\V{f+h}_X
=\V{\Phi(f)+\Phi(h)}_Y
=\V{\Phi(f)+v}_Y=2,
\]
and hence $h\in\GX{f}$.
This implies $v=\Phi(h)\in\Phi(\GX{f})$.
Since $v\in\GY{\Phi(f)}$ was arbitrary,
we conclude that
$\GY{\Phi(f)}\subset\Phi(\GX{f})$.
Therefore $\Phi(\GX{f})=\GY{\Phi(f)}$.
\end{proof}

We next establish several preliminary lemmas needed for the proof of
Theorem~\ref{Thm:auxiliary representation theorem}.

\begin{lem}\label{lem:pointwise norm estimate}
For every $f\in\SCX$, $x_0\in X$, and $\epsilon>0$, there exists $h\in\SCX$ such that
$h(x_0)=1$ and
$\V{f+h}_X
<
f(x_0)+\epsilon+1$.
\end{lem}

\begin{proof}
Define
$O:=\{x\in X:f(x)<f(x_0)+\epsilon\}$.
Then $O$ is an open neighborhood of $x_0$.
By Urysohn's lemma, there exists
$h\in \SCX$ such that
$h(x_0)=1$ and $h=0$ on $X\setminus O$.
We claim that
\[
\VX{f+h}<f(x_0)+\epsilon+1.
\]
Let $x\in X$ be arbitrary.
Since $f+h$ attains its maximum
on the compact space $X$,
it is enough to prove $f(x)+h(x)<f(x_0)+\epsilon+1$.
If $x\in O$, then
$f(x)<f(x_0)+\epsilon$.
Since $h(x)\le1$, we obtain
$f(x)+h(x)<f(x_0)+\epsilon+1$.
If $x\notin O$, then
$h(x)=0$
and $f(x)\leq1$. 
Hence
$f(x)+h(x)<f(x_0)+\epsilon+1$.
Therefore
$f(x)+h(x)<f(x_0)+\epsilon+1$
for all $x\in X$
as desired.
\end{proof}

As a first consequence of Lemma~\ref{lem:pointwise norm estimate},
we obtain the injectivity of $\Phi$.

\begin{lem}\label{lem:injective}
The map $\Phi$ is injective.
\end{lem}

\begin{proof}
Let $f,g\in \SCX$ be such that $f\neq g$. Then there exists
$x_0\in X$ such that $f(x_0)\neq g(x_0)$.
Without loss of generality, we may assume that
$f(x_0)<g(x_0)$.
Set
\[
\epsilon=\frac{g(x_0)-f(x_0)}2.
\]
Then $\epsilon>0$.
By Lemma~\ref{lem:pointwise norm estimate}, there exists $h\in \SCX$ such that $h(x_0)=1$ and  
\[
\V{f+h}_X
<f(x_0)+\epsilon+1.
\]
By the definition of $\epsilon$,
$g(x_0)=f(x_0)+2\epsilon$.
Since $h(x_0)=1$, we have 
\[
f(x_0)+2\epsilon+1
=g(x_0)+h(x_0)
\leq\V{g+h}_X.
\]
Combining the preceding inequalities gives
$\V{f+h}_X<f(x_0)+\epsilon+1<\V{g+h}_X$.
Since $\Phi$ is norm-additive, we obtain
\[
\V{\Phi(f)+\Phi(h)}_Y
=\V{f+h}_X
<\V{g+h}_X
=\V{\Phi(g)+\Phi(h)}_Y.
\]
This implies $\Phi(f)\neq\Phi(g)$,
proving the injectivity of $\Phi$.
\end{proof}

By Lemma~\ref{lem:injective}, $\Phi$ is a bijective norm-additive map.
Thus $\Phii$ is also a norm-additive map.
To extract information about the underlying compact spaces from
norm-additivity, we next examine how the condition
$\V{f+g}_X=2$ is related to the points at which functions
attain their norm.

For each $f\in \SCK$, define
\[
M_K(f):=\{x\in K:f(x)=1\}.
\]
We call $M_K(f)$ the \emph{maximum set} of $f$.

We next clarify the relation between the maximum sets $M_K(f)$
and the star sets $\GK{f}$.

\begin{lem}\label{lem:maximum sets and star sets}
Let $f,g \in \SCK$. Then $\MK{g}\subset\MK{f}$ if and only if $\GK{g} \subset\GK{f}$.
\end{lem}

\begin{proof}
Suppose that $\MK{g}\subset \MK{f}$. We prove that $\GK{g}\subset \GK{f}$.
Fix $h\in \GK{g}$ arbitrarily. Then $\VK{g+h}=2$, so there exists $x_0\in K$ such that
\[
g(x_0)+h(x_0)=2.
\]
Since $g,h\in \SCK$, we obtain $g(x_0)=h(x_0)=1$. Thus $x_0\in \MK{g}$.
Because $\MK{g}\subset \MK{f}$, we have $x_0\in \MK{f}$, and hence $f(x_0)=1$. Therefore
\[
2=f(x_0)+h(x_0)\le \VK{f+h}\le 2.
\]
Hence $\VK{f+h}=2$, which shows that $h\in \GK{f}$. Since $h\in \GK{g}$ was arbitrary, we conclude that $\GK{g}\subset \GK{f}$.

Conversely, suppose that $\MK{g}\not\subset \MK{f}$. We prove that $\GK{g}\not\subset \GK{f}$.
Choose $x_1\in \MK{g}\setminus \MK{f}$. Then $f(x_1)<1=g(x_1)$.
By Urysohn's lemma, there exists $h_1\in \SCK$ such that $h_1(x_1)=1$ and $h_1=0$ on $\MK{f}$.
Then
\[
2=g(x_1)+h_1(x_1)\le \VK{g+h_1}\le 2,
\]
and thus $\V{g+h_1}_K=2$.
Hence $h_1\in \GK{g}$.
We now show that $h_1\notin \GK{f}$.
That is, $\V{f+h_1}_K<2$.
Let $x\in K$ be arbitrary.
If $x\in \MK{f}$, then
$h_1(x)=0$, which yields
$f(x)+h_1(x)=f(x)<2$.
If $x\in K\setminus\MK{f}$,
then $f(x)<1$ and $h_1(x)\leq1$,
which shows that
$f(x)+h_1(x)<2$.
Therefore $f(x)+h_1(x)<2$ for every $x\in K$.
Thus $\V{f+h_1}_K<2$,
and so $h_1\notin \GK{f}$.
Since $h_1\in \GK{g}$, we conclude that $\GK{g}\not\subset \GK{f}$.
This proves that $\GK{g}\subset \GK{f}$ implies $\MK{g}\subset \MK{f}$.
\end{proof}

The preceding characterization allows us to show that
inclusions of maximum sets are preserved under $\Phi$.

\begin{lem}\label{lem:inclusion of maximum sets}
Let $f,g\in\SCX$.
If $\MX{g}\subset\MX{f}$, then
$\MY{\Phi(g)}\subset\MY{\Phi(f)}$.
\end{lem}

\begin{proof}
Suppose that $\MX{g}\subset\MX{f}$.
By Lemma~\ref{lem:maximum sets and star sets}, we have
$\GX{g}\subset\GX{f}$.
Hence,
\[
\Phi(\GX{g})\subset\Phi(\GX{f}).
\]
By Lemma~\ref{lem:star-set-preservation}
we obtain
$\Phi(\GX{h})=\GY{\Phi(h)}$
for $h\in\SCX$.
Combining this with the inclusion above, we obtain
$\GY{\Phi(g)}\subset\GY{\Phi(f)}$.
Therefore Lemma~\ref{lem:maximum sets and star sets} yields
$\MY{\Phi(g)}\subset \MY{\Phi(f)}$.
\end{proof}

Having established the behavior of maximum sets under $\Phi$,
we now use this information to recover the points of the underlying
compact spaces.

For each $x_0\in K$, define
\[
P_K(x_0):=\{f\in \SCK:f(x_0)=1\}.
\]

\begin{lem}\label{lem:peak sets}
Let $x_0,x_1\in\K$.
If $\PK{x_0}\subset\PK{x_1}$,
then $x_0=x_1$.
\end{lem}

\begin{proof}
Suppose that $x_0\neq x_1$.
By Urysohn's lemma, there exists
$f_0\in\SCK$ such that
$f_0(x_0)=1$ and $f_0(x_1)=0$.
Hence
$f_0\in\PK{x_0}$, but
$f_0\notin\PK{x_1}$.
Therefore
$f_0\in\PK{x_0}\setminus\PK{x_1}$,
which implies that
$\PK{x_0}\not\subset\PK{x_1}$.
This shows that $\PK{x_0}\subset\PK{x_1}$
implies $x_0=x_1$.
\end{proof}

We now aim to associate with each point $x_0\in X$ a corresponding
point of $Y$ through the map $\Phi$.

\begin{lem}\label{lem:intersection of maximum sets}
The intersection
\[
\bigcap_{f\in\PX{x_0}}\MY{\Phi(f)}
\]
is non-empty for each $x_0\in X$.
\end{lem}

\begin{proof}
Fix $x_0\in X$. We show that the family
$\{\MY{\Phi(f)}: f\in\PX{x_0}\}$
has the finite intersection property. Let $f_1,\dots,f_n\in\PX{x_0}$ be arbitrary, and define
\[
g:=\frac1n\sum_{j=1}^n f_j.
\]
Since $0\le f_j\le1$ and $f_j(x_0)=1$ for all $1\le j\le n$, we have
$0\le g\le1$ and $g(x_0)=1$. Hence $g\in\SCX$.
We claim that
\[
\MX{g}\subset \MX{f_j}
\qquad (1\le j\le n).
\]
Indeed, let $x\in\MX{g}$. Then
\[
1=g(x)=\frac1n\sum_{j=1}^n f_j(x).
\]
Since $0\le f_j(x)\le1$, we obtain
$f_j(x)=1$ for every
$1\le j\le n$.
Thus $x\in\MX{f_j}$ for all $1\le j\le n$.
Since $x\in\MX{g}$ was arbitrary,
we conclude $\MX{g}\subset\MX{f_j}$
for every $1\leq j\leq n$ as desired.

By the result proved in the preceding paragraph,
we may apply Lemma~\ref{lem:inclusion of maximum sets}
to obtain
\[
\MY{\Phi(g)}\subset\MY{\Phi(f_j)}
\qquad (1\le j\le n).
\]
Therefore
\[
\MY{\Phi(g)}\subset \bigcap_{j=1}^n \MY{\Phi(f_j)}.
\]
Since $\Phi(g)\in \SCY$, we have
$\MY{\Phi(g)}\neq\emptyset$. Hence
$\bigcap_{j=1}^n \MY{\Phi(f_j)}\neq\emptyset$.
This shows that the family $\{\MY{\Phi(f)}: f\in\PX{x_0}\}$ has the finite intersection property. Because each $\MY{\Phi(f)}$ is closed and $Y$ is compact, we conclude that
$\bigcap_{f\in\PX{x_0}}\MY{\Phi(f)}\neq\emptyset$.
\end{proof}

We are now in a position to construct a correspondence between the
underlying points of $X$ and $Y$ through $\Phi$.

\begin{lem}\label{lem3.9}
For each $\xz\in\X$ there exists a unique
$\yz\in\Y$ such that
\[
\Phi(\PX{\xz})=\PY{\yz}.
\]
\end{lem}

\begin{proof}
Fix an arbitrary $\xz\in\X$.
By Lemma~\ref{lem:intersection of maximum sets}, we have
$\bigcap_{f\in\PX{\xz}}\MY{\Phi(f)}\neq\emptyset$.
Choose
\[
\yz\in\bigcap_{f\in\PX{\xz}}\MY{\Phi(f)}.
\]
We first prove that
\begin{equation}\label{lem3.9.1}
\Phi(\PX{\xz})\subset\PY{\yz}.
\end{equation}
Fix $u\in\Phi(\PX{\xz})$.
There exists $f_0\in\PX{\xz}$ such that
$u=\Phi(f_0)$.
By the choice of $\yz$,
we have $\yz\in\MY{\Phi(f_0)}$.
Then $\Phi(f_0)(\yz)=1$,
which yields $u=\Phi(f_0)\in\PY{\yz}$.
Since $u\in\Phi(\PX{\xz})$ was arbitrary,
we conclude that $\Phi(\PX{\xz})\subset\PY{\yz}$,
which proves \eqref{lem3.9.1}.

Applying the same argument to the map $\Phii$ and the point $\yz$,
we obtain
$\Phii(\PY{\yz})\subset\PX{\xo}$
for some $\xo\in\X$.
Applying $\Phi$ to both sides, we obtain
\begin{equation*}
\PY{\yz}\subset\Phi(\PX{\xo}).
\end{equation*}
Combining \eqref{lem3.9.1} with the last inclusion, we have
\[
\Phi(\PX{\xz})
\subset
\PY{\yz}
\subset
\Phi(\PX{\xo}).
\]
Since $\Phi$ is injective, it follows that
$\PX{\xz}\subset\PX{\xo}$.
By Lemma~\ref{lem:peak sets}, we obtain
$\xz=\xo$.
Hence the preceding inclusions imply
$\Phi(\PX{\xz})=\PY{\yz}$.

Finally, we prove the uniqueness of $\yz$.
Suppose that another point $\yo\in\Y$ satisfies
$\Phi(\PX{\xz})=\PY{\yo}$.
Then
$\PY{\yz}=\Phi(\PX{\xz})=\PY{\yo}$.
Applying Lemma~\ref{lem:peak sets}, we obtain
$\yz=\yo$.
Therefore such a point $\yz\in\Y$ is unique.
\end{proof}

We now use the point correspondence obtained above to derive the
representation of $\Phi$ and complete the proof of
Theorem~\ref{Thm:auxiliary representation theorem}.

\begin{proof}[\textbf{Proof of Theorem~\ref{Thm:auxiliary representation theorem}}]
Let $x\in\X$.
By Lemma~\ref{lem3.9}, there exists a unique $y\in\Y$
such that $\Phi(\PX{x})=\PY{y}$.
Set $\varphi(x)=y$.
Then $\varphi\colon\X\to\Y$ is a map
that satisfies
\begin{equation}\label{eq:phi}
\Phi(\PX{x})=\PY{\varphi(x)}
\qquad (x\in X).
\end{equation}
Similarly, applying Lemma~\ref{lem3.9} to $\Phi^{-1}$, we obtain a map
$\tau\colon Y\to X$
such that
\begin{equation}\label{eq:tau}
\Phi^{-1}(\PY{y})=\PX{\tau(y)}
\qquad (y\in Y).
\end{equation}

We prove that $\varphi$ and $\tau$ are
both bijective with $\tau=\varphi^{-1}$.
For each $x\in X$,
it follows from \eqref{eq:phi}
and \eqref{eq:tau} that
\[
\PX{x}
=
\Phi^{-1}(\PY{\varphi(x)})
=
\PX{\tau(\varphi(x))}.
\]
Hence, by Lemma~\ref{lem:peak sets},
$x=\tau(\varphi(x))$ for every $x\in X$.
Similarly,
\eqref{eq:phi} and \eqref{eq:tau} show that
\[
\PY{y}
=
\Phi(\PX{\tau(y)})
=
\PY{\varphi(\tau(y))}
\]
for every $y\in Y$.
Applying Lemma~\ref{lem:peak sets} once again, we obtain
$y=\varphi(\tau(y))$ for every $y\in Y$.
Consequently, $\varphi$ and $\tau$ are bijections satisfying
$\tau=\varphi^{-1}$.

We next prove
\begin{equation}\label{eq:mainformula}
\Phi(f)(y)=f(\tau(y))
\qquad (f\in\SCX,\ y\in Y).
\end{equation}
Fix $f\in\SCX$, $y\in Y$, and $\epsilon>0$.
By Lemma~\ref{lem:pointwise norm estimate}, there exists $h\in \SCX$ such that 
$h(\tau(y))=1$ and
\[
\VX{f+h}<f(\tau(y))+\epsilon+1.
\]
Since $h(\tau(y))=1$, we have
$h\in P_X(\tau(y))$.
Since $\tau=\varphi^{-1}$, we have
$\varphi(\tau(y))=y$.
Hence, by \eqref{eq:phi},
$\Phi(h)\in P_Y(y)$.
Thus
$\Phi(h)(y)=1$.
Therefore we have
\[
\Phi(f)(y)+1
=
\Phi(f)(y)+\Phi(h)(y)
\le
\VY{\Phi(f)+\Phi(h)}.
\]
Since $\Phi$ is norm-additive,
we obtain
\[
\V{\Phi(f)+\Phi(h)}_Y=\V{f+h}_X
<f(\tau(y))+\epsilon+1.
\]
Combining this with the previous inequality, we obtain
$\Phi(f)(y)+1
<f(\tau(y))+\epsilon+1$.
Hence
$\Phi(f)(y)
<f(\tau(y))+\epsilon$.
Since $\epsilon>0$ was arbitrary, we conclude that
$\Phi(f)(y)\le f(\tau(y))$.
Applying the same argument to
$\Phi^{-1}$,
we obtain
$\Phi^{-1}(u)(x)\leq u(\varphi(x))$
for $u\in\SCY$ and $x\in X$.
Substituting $u=\Phi(f)$ and $x=\tau(y)$
into the last inequality gives
$f(\tau(y))\leq\Phi(f)(y)$,
where we have used that $\tau=\varphi^{-1}$.
Hence \eqref{eq:mainformula} holds.

It remains to prove that $\tau$ is continuous.
Let $O\subset X$ be open, and fix $y_0\in\tau^{-1}(O)$.
Since $\tau(y_0)\in O$ and $O$ is open, by Urysohn's lemma there exists
$f_0\in\SCX$ such that
\[
f_0(\tau(y_0))=1
\qquad\text{and}\qquad
f_0=0 \ \text{on }X\setminus O.
\]
Define
\[
V:=\left\{y\in Y:\Phi(f_0)(y)>\frac{1}{2}\right\}.
\]
Since $\Phi(f_0)$ is continuous, $V$ is open.
Moreover,
$\Phi(f_0)(y_0)=f_0(\tau(y_0))=1$
by \eqref{eq:mainformula}, so
$y_0\in V$.
We claim that
$V\subset\tau^{-1}(O)$.
Indeed, fix $y\in V$ arbitrarily.
Then
$\Phi(f_0)(y)>1/2$.
By \eqref{eq:mainformula}, we have
$f_0(\tau(y))=\Phi(f_0)(y)>1/2$.
Since $f_0=0$ on $X\setminus O$, we obtain
$\tau(y)\in O$.
Hence
$y\in\tau^{-1}(O)$.
Since $y\in V$ was arbitrary, we conclude that
$V\subset\tau^{-1}(O)$.
Thus $V$ is an open neighborhood of $y_0$
contained in $\tau^{-1}(O)$.
Since $y_0\in\tau^{-1}(O)$ was arbitrary,
we conclude that $\tau^{-1}(O)$ is open.
Hence $\tau$ is continuous.

Since $\tau$ is a continuous bijection from the compact space $Y$
onto the Hausdorff space $X$, it follows that
$\tau$ is a homeomorphism.
This completes the proof of Theorem~\ref{Thm:auxiliary representation theorem}.
\end{proof}

\subsection{The positive unit ball case}

We now establish the main result of this subsection, namely,
the corresponding representation theorem for the positive unit ball.
Set
\[
\BK:=\set{f\in C(K):f\geq0,\ \VK{f}\leq1}.
\]

\begin{thm}\label{Thm:auxiliary representation theorem ball}
Let $X$ and $Y$ be compact Hausdorff spaces.
Let $\Psi\colon\BX\to\BY$ be a surjective norm-additive map.
Then there exists a homeomorphism $\tau\colon Y\to X$ such that
$\Psi(f)(y)=f(\tau(y))$
for all $f\in\BX$ and $y\in Y$.
\end{thm}

Throughout the remainder of this subsection, let $X$ and $Y$ be
compact Hausdorff spaces, and let
$\Psi\colon\BX\to\BY$ be a surjective norm-additive map.
That is,
\[
\VY{\Psi(f)+\Psi(g)}=\VX{f+g}
\qquad(f,g\in\BX).
\]

We first observe that $\Psi$ preserves norms.

\begin{lem}\label{lem:ball-norm-preserving}
For every $f\in\BX$, we have
$\VY{\Psi(f)}=\VX{f}$. 
\end{lem}

\begin{proof}
Let $f\in\BX$. Since $\Psi$ is norm-additive,
$\VY{\Psi(f)+\Psi(f)}=\VX{f+f}$.
Hence
$\VY{\Psi(f)}=\VX{f}$.
\end{proof}

We now reduce the problem on the positive unit ball to the
positive unit sphere by a normalization argument.
For each $0<r\leq1$, define
$\Psir\colon\SCX\to\SCY$
by
\[
\Psir(f):=\frac{1}{r}\Psi(rf)
\qquad (f\in\SCX).
\]
By Lemma~\ref{lem:ball-norm-preserving},
$\|\Psi(rf)\|_Y=r$.
Since $\Psi(rf)\geq0$, it follows that
$\Psi(rf)/r\in S_Y^+$.
Thus $\Psi_r$ is well defined.
We next show that each $\Psir$ satisfies the hypotheses of
Theorem~\ref{Thm:auxiliary representation theorem}.

\begin{lem}\label{Psi is a surjective norm-additive}
For every $0<r\leq1$, the map
$\Psir\colon\SCX\to\SCY$
is a surjective norm-additive map.
\end{lem}

\begin{proof}
First, we show that $\Psir$ is surjective.
Let $u\in\SCY$.
Since $0<r\leq1$, we have 
$\V{ru}_Y=r$, and thus $ru\in\BY$.
Since $\Psi\colon\BX\to\BY$ is surjective, there exists $f_0\in\BX$ such that
$\Psi(f_0)=ru$.
By Lemma~\ref{lem:ball-norm-preserving},
we obtain
$\VX{f_0}=\VY{\Psi(f_0)}=r$.
Let $g_0=(1/r)f_0$.
Then $g_0\in\SCX$.
Since $rg_0=f_0$ and $\Psi(f_0)=ru$, we obtain
$\Psir(g_0)
=\Psi(rg_0)/r=\Psi(f_0)/r
=u$.
Therefore $\Psir$ is surjective.

Next, we show that $\Psir$ is norm-additive.
Let $f,g\in \SCX$ be arbitrary.
Since $\Psi$ is norm-additive,
\[
\V{\Psir(f)+\Psir(g)}_Y
=\frac{1}{r}\V{\Psi(rf)+\Psi(rg)}_Y
=\frac{1}{r}\V{rf+rg}_X
=\V{f+g}_X.
\]
Therefore $\Psir$ is norm-additive.
\end{proof}

By the preceding lemma, for each $0<r\leq1$, the map
$\Psir\colon\SCX\to\SCY$
is a surjective norm-additive map.
Hence, by Theorem~\ref{Thm:auxiliary representation theorem},
there exists a homeomorphism
$\tau_r\colon Y\to X$
such that
\[
\Psir(f)=f\circ\tau_r
\qquad (f\in\SCX).
\]

We next translate this representation back to $\Psi$.
Fix $0<r\leq1$, and let $f\in\BX$ satisfy $\VX{f}=r$.
Then $(1/r)f\in\SCX$, and hence, by
the preceding equality,
\[
\Psir\left(\frac{1}{r}f\right)
=
\left(\frac{1}{r}f\right)\circ\tau_r.
\]
By the definition of $\Psir$,
\[
\Psir\left(\frac{1}{r}\,f\right)
=
\frac{1}{r}\Psi(f).
\]
Combining the above equalities, we obtain
\begin{equation}\label{eq:ball-representation-r}
\Psi(f)=f\circ\tau_r
\qquad
(f\in\BX,\ \VX{f}=r).
\end{equation}
Since $0<r\leq1$ was arbitrary,
\eqref{eq:ball-representation-r} holds for every $0<r\leq1$.

Hereafter, $\mathbf{0}$ and $\mathbf{1}$ denote the constant functions with values $0$ and $1$, respectively, on the underlying space whenever the domain is clear.
We now use the representation obtained above for functions of fixed norm to
show that the homeomorphism $\tau_r$ is independent of $r$ and complete the
proof of Theorem~\ref{Thm:auxiliary representation theorem ball}.

\begin{proof}[\textbf{Proof of Theorem~\ref{Thm:auxiliary representation theorem ball}.}]
Let $0<r\leq1$.
First we prove that $\tau_r=\tau_1$,
where $\tau_r$ is the homeomorphism from
$\Y$ onto $\X$ as in \eqref{eq:ball-representation-r}.
The assertion is trivial for $r=1$.
Fix $0<r<1$.
Suppose, to the contrary, that $\tau_r\neq\tau_1$.
Then there exists $y_0\in Y$ such that
$\tau_r(y_0)\neq\tau_1(y_0)$.
By Urysohn's lemma, there exists $f\in \SCX$ such that
$f(\tau_1(y_0))=1$ and
$f(\tau_r(y_0))=0$.
Applying \eqref{eq:ball-representation-r},
\[
\Psi(f)(y_0)
=
f(\tau_1(y_0))
=
1.
\]
Put $g:=r(\mathbf{1}-f)$.
Since $r>0$ and $f\in \SCX$, we have $0\le g\le r$. 
Since $f(\tau_r(y_0))=0$, we obtain $g(\tau_r(y_0))=r$.
Therefore $g\in \BX$ and $\VX{g}=r$.
By \eqref{eq:ball-representation-r}, we have
\[
\Psi(g)(y_0)
=
g(\tau_r(y_0))
=
r.
\]
Combining this with the preceding identity gives
\[
1+r
=
\Psi(f)(y_0)+\Psi(g)(y_0)
\leq
\VY{\Psi(f)+\Psi(g)}.
\]
Since $\Psi$ is norm-additive,
we conclude
$1+r\leq\V{f+g}_X$.
On the other hand, since $g=r(\mathbf{1}-f)$,
we have $\V{f+g}_X=\V{(1-r)f+r\mathbf{1}}_X\leq1$.
Combining this with the last inequality gives
$1+r\leq\V{f+g}_X\leq1$, which yields
$r\leq0$, a contradiction.
We conclude that $\tau_r=\tau_1$.

Set $\tau=\tau_1$.
Then $\tau\colon\Y\to\X$ is a homeomorphism
with $\tau=\tau_r$ for every $0<r\leq1$
by the result proved in the preceding paragraph.
We prove that $\Psi(f)=f\circ\tau$ for every
$f\in\BX$.
Let $f\in\BX\setminus\{0\}$.
Put $r:=\V{f}_X$, and then
$0<r\leq1$.
Applying \eqref{eq:ball-representation-r}
with $\tau_r=\tau$,
we have $\Psi(f)=f\circ\tau$.
Now we consider the case when $f=\mathbf{0}$.
Lemma~\ref{lem:ball-norm-preserving} shows that
$
\V{\Psi(\mathbf{0})}_Y
=
\V{\mathbf{0}}_X
=
0.
$
Hence $\Psi(\mathbf{0})=\mathbf{0}$, and therefore
$\Psi(\mathbf{0})=\mathbf{0}\circ\tau$.
Thus
$\Psi(f)=f\circ\tau$
for every $f\in\BX$.
\end{proof}

\section{The case $1<p<\infty$}\label{section3}

We now prove the forward implication of
Theorem~\ref{thm1.1} for $1<p<\infty$.
Throughout this section, let $1<p<\infty$, and set
\begin{align*}
\Sp:=\{f\in C^1(I):f(0)\geq0,\ f'\geq0,\ \Vp{f}=1\},\\
\Vp{f}:=\bigl(|f(0)|^p+\VI{f'}^p\bigr)^{1/p}
\qquad(f\in C^1(I)).
\end{align*}
Let $T\colon\Sp\to\Sp$ be a surjective norm-additive map; that is,
\[
\Vp{T(f)+T(g)}=\Vp{f+g}
\qquad (f,g\in\Sp).
\]
Our strategy is to transfer $T$ to a surjective norm-additive map on
\[
\B=\{u\in C(I):u\geq0,\ \VI{u}\leq1\}
\]
via the derivative parametrization and then apply
Theorem~\ref{Thm:auxiliary representation theorem ball}.
To carry out this reduction, we define a map
$D_p\colon\Sp\to\B$ by
\[
D_p(f)=f'
\qquad(f\in\Sp).
\]

\subsection{The derivative parametrization}

The following lemma is the key ingredient for this reduction, since it shows
that $D_p$ is a bijection between $\mathcal{S}_{p}^{+}$ and $\B$.

\begin{lem}\label{lem:[p] derivative map}
The map $D_p$ is bijective with its inverse
\[
D_p^{-1}(u)(t)
=
\bigl(1-\|u\|_{I}^p\bigr)^{1/p}
+
\int_0^t u(s)\,ds
\]
for all $u\in\B$ and $t\in I$.
\end{lem}

\begin{proof}
Define
\[
\Jp(u)(t)
:=
\bigl(1-\|u\|_I^p\bigr)^{1/p}
+
\int_0^t u(s)\,ds
\qquad(u\in\B,\ t\in I).
\]
Fix an arbitrary $u\in \B$.
We prove that $\Jp(u)\in\Sp$.
By definition, $\Jp(u)\in C^1(I)$.
Since $u\in \B$,
we have $u\geq0$, and $\|u\|_I\le1$.
Hence
\[
\Jp(u)(0)
=
\bigl(1-\|u\|_I^p\bigr)^{1/p}
\ge0,
\qquad
\Jp(u)'=u\ge0.
\]
Moreover,
\[
\|\Jp(u)\|_{(p)}^p
=
\Jp(u)(0)^p+\|\Jp(u)'\|_I^p
=
1-\|u\|_I^p+\|u\|_I^p
=
1.
\]
These show that $\Jp(u)\in\Sp$.
Therefore,
$\Jp\colon\B\to\Sp$
is well defined.

We prove that $D_p$ is bijective
and $\Jp$ is its inverse. Since
$D_p(J_p(u))
=\Jp(u)'
=u$ for all $u\in\B$,
we obtain
$D_p\circ\Jp=\operatorname{id}_{\B}$.
Next, we show that $\Jp\circ D_p=\operatorname{id}_{\Sp}$.
Fix an arbitrary $f\in \Sp$.
Since $D_p(f)=f'$, we have
\[
\Jp(D_p(f))(t)
=
\bigl(1-\|f'\|_I^p\bigr)^{1/p}
+\int_0^t f'(s)\,ds\qquad (t\in I).
\]
Since $f\in \Sp$,
we have
$f(0)^p+\VI{f'}^p=1$.
Hence the above equalities and the Fundamental Theorem of Calculus yield
\[
\Jp(D_p(f))(t)
=
f(0)+\int_0^t f'(s)\,ds
=
f(t)
\qquad (t\in I).
\]
Therefore, 
$\Jp\circ D_p=\operatorname{id}_{\Sp}$.
It follows that $D_p$ is bijective and its inverse is given by
$D_p^{-1}=\Jp$.
\end{proof}

\subsection{Preservation of the initial value}

To obtain a norm-additive map on $\B$ via the
derivative parametrization, we first establish that $T$ preserves
the initial value.

Define
\begin{equation}\label{eq:alpha}
\alpha(t):=(1+t)^p+1-t^p \qquad(t\in I).
\end{equation}
Since $p>1$, we have
\[
\alpha'(t)=p(1+t)^{p-1}-pt^{p-1}>0
\]
for all $t\in I$. 
Then $\alpha$ is strictly increasing, and hence injective.

Let $f\in\Sp$.
Then $f(0)^p+\VI{f'}^p=\Vp{f}^p=1$.
Hence
\[
\VI{f'}^p=1-f(0)^p.
\]
Combining this with \eqref{eq:alpha} gives
\[
\Vp{\mathbf{1}+f}^p
=
(1+f(0))^p+\|f'\|_I^p
=(1+f(0))^p+1-f(0)^p
=\alpha(f(0)).
\]
Thus we obtain
\begin{equation}\label{lem2.2.1}
\alpha(f(0))=\Vp{\mathbf{1}+f}^p
\qquad(f\in\Sp).
\end{equation}

For each $f\in\Sp$, we define the \emph{star set} of $f$ in $\Sp$ by
\[
\G(f):=\{g\in\Sp:\Vp{f+g}=2\}.
\]
As a first step, we seek a characterization of $\mathbf{1}$
that is preserved by $T$.

\begin{lem}\label{lem:G(1)={1}}
$\G(\mathbf 1)=\{\mathbf 1\}$.
\end{lem}

\begin{proof}
First, we show that $\G(\mathbf1)\subset\{\mathbf1\}$.
Fix an arbitrary $f\in\G(\mathbf1)$.
Then,
by definition,
\[
\Vp{\mathbf1+f}=2.
\]
Let $\alpha$ be the function
defined as in \eqref{eq:alpha}.
We have
$\alpha(1)=2^p$.
Combining these equalities
with \eqref{lem2.2.1},
\[
\alpha(f(0))=\Vp{\mathbf{1}+f}^p
=2^p=\alpha(1).
\]
Since $\alpha$ is injective,
we conclude that $f(0)=1$.
By the definition of $\Vp{\cdot}$, we have
\[
\VI{f'}^p=1-f(0)^p=0.
\]
Hence $\VI{f'}=0$, and therefore 
$f'=\mathbf{0}$.
Since $f(0)=1$, 
we conclude that
$f=\mathbf1$.
Since $f\in\G(\mathbf{1})$ was arbitrary,
$\G(\mathbf1)\subset\{\mathbf1\}$.

Conversely, by the definition
of $\Vp{\cdot}$, we have
$\Vp{\mathbf1+\mathbf1}=2$.
Thus $\mathbf1\in\G(\mathbf1)$,
and hence $\{\mathbf{1}\}\subset\G(\mathbf{1})$.
Therefore $\G(\mathbf1)=\{\mathbf1\}$.
\end{proof}

To characterize $\mathbf{1}$ solely in terms of its star set,
we now establish the converse implication.

\begin{lem}\label{lem:singleton-star-set}
For each $f\in\Sp\setminus\{\mathbf{1}\}$
there exists $h\in\Sp$ such that $h\in\G(f)$
and $h\neq f$.
\end{lem}

\begin{proof}
Let $f\in\Sp\setminus\{\mathbf{1}\}$.
Since $f(0)^p+\V{f'}_{I}^p=\Vp{f}^p=1$,
we have $0\leq f(0)\leq 1$.
We prove that $f(0)<1$.
Suppose, to the contrary, that $f(0)=1$.
Then $\V{f'}_{I}=0$, and hence $f'=\mathbf{0}$. Therefore
$f=\mathbf{1}$, which contradicts our assumption.
Thus $f(0)<1$.
Hence $\VI{f'}>0$.
Choose \(t_0\in I\) such that \(f'(t_0)=\VI{f'}\).
By the continuity of \(f'\), there exists
\(t_1\in I\setminus\{t_0\}\) such that \(f'(t_1)>0\).
Choose $u\in C(I)$ satisfying
\[
0\leq u\leq1,\qquad
u(t_0)=1,\qquad
u(t_1)=0.
\]
Put $v:=uf'$, and then $v\in C(I)$
and $v\geq0$.
Since $u(t_0)=1=\VI{u}$ and $f'(t_0)=\VI{f'}$,
we have
\[
v(t_0)=\VI{f'}=\VI{v}.
\]
Define $h\in C^1(I)$ by
\[
h(t):=f(0)+\int_0^tv(s)\,ds
\qquad(t\in I).
\]
Then $h(0)=f(0)$ and $h'=v$.
Since $u(t_1)=0$, we obtain
$v(t_1)=0<f'(t_1)$.
This shows that $h'=v\neq f'$,
and hence $h\neq f$.
Moreover, by the defining property
of $\Vp{\cdot}$ with $\VI{v}=\VI{f'}$,
\[
\Vp{h}^p
=f(0)^p+\VI{v}^p
=f(0)^p+\VI{f'}^p
=\Vp{f}^p=1.
\]
Thus $h\in\Sp$ with $h\neq f$.

Finally we prove $h\in\G(f)$.
Since $v(t_0)=\VI{f'}=f'(t_0)$,
we have $\VI{f'+v}=2\VI{f'}$.
Since $\Vp{f}=1$, we obtain
\[
\Vp{f+h}^p
=(2f(0))^p+\VI{f'+v}^p
=2^p\bigl(f(0)^p+\VI{f'}^p\bigr)
=2^p.
\]
Therefore $\Vp{f+h}=2$,
which yields $h\in\G(f)$.
\end{proof}

We now use this characterization to determine the image of
$\mathbf{1}$ under $T$.

\begin{lem}\label{lem:T(1)=1}
$T(\mathbf{1})=\mathbf{1}$.
\end{lem}

\begin{proof}
First we prove $\G(T(\mathbf{1}))=\{T(\mathbf{1})\}$.
By Lemma~\ref{lem:G(1)={1}}, we have $\G(\mathbf{1})=\{\mathbf{1}\}$. Hence
$T(\G(\mathbf{1}))=\{T(\mathbf{1})\}$.
On the other hand, the same argument as in the proof of
Lemma~\ref{lem:star-set-preservation} can be applied to
$T\colon\Sp\to\Sp$,
in place of $\Phi\colon\SCX\to\SCY$ to conclude
$T(\G(\mathbf{1}))=\G(T(\mathbf{1}))$:
In fact, 
let $f\in T(\G(\mathbf{1}))$ be arbitrary.
Then there exists $g_1\in \G(\mathbf{1})$ such that $T(g_1)=f$.
By definition of $\G(\mathbf{1})$, we have $\Vp{\mathbf{1}+g_1}=2$.
Since $T$ is norm-additive,
\[
\Vp{T(\mathbf{1})+f}
=
\Vp{T(\mathbf{1})+T(g_1)}
=
\Vp{\mathbf{1}+g_1}=2.
\]
Thus $f\in\G(T(\mathbf{1}))$. 
Since $f\in T(\G(\mathbf{1}))$ was arbitrary, we conclude that $T(\G(\mathbf{1}))\subset\G(T(\mathbf{1}))$.
Conversely, let $f\in\G(T(\mathbf{1}))$. By definition, we have $\Vp{T(\mathbf{1})+f}=2$. 
Since $T\colon\Sp\to\Sp$ is surjective,
there exists $g_2\in\Sp$ such that $T(g_2)=f$.
Since $T$ is norm-additive,
we obtain
$\Vp{\mathbf{1}+g_2}
=
\Vp{T(\mathbf{1})+T(g_2)}
=
\Vp{T(\mathbf{1})+f}
=
2$,
which yields $g_2\in\G(\mathbf{1})$.
Thus $f=T(g_2)\in T(\G(\mathbf{1}))$. 
Since $f\in\G(T(\mathbf{1}))$ was arbitrary, we conclude that
$\G(T(\mathbf{1}))\subset T(\G(\mathbf{1}))$.
Therefore,
$\G(T(\mathbf{1}))
=T(\G(\mathbf{1}))=\{T(\mathbf{1})\}$.

We prove $T(\mathbf{1})=\mathbf{1}$.
Suppose, to the contrary, that $T(\mathbf{1})\neq\mathbf{1}$.
By Lemma~\ref{lem:singleton-star-set},
applied to $T(\mathbf{1})$, in place of $f$,
there exists $h\in\Sp$ such that
$h\in\G(T(\mathbf{1}))$ and $h\neq T(\mathbf{1})$.
This is a contradiction with
$\G(T(\mathbf{1}))=\{T(\mathbf{1})\}$.
Therefore
we conclude that $T(\mathbf{1})=\mathbf{1}$.
\end{proof}

Having identified $\mathbf{1}$ as a fixed point of $T$, we now
return to our goal of understanding the behavior of the initial
value under $T$.

\begin{lem}\label{lem:T(f)(0)=f(0)}
For every $f\in\Sp$, we have
$T(f)(0)=f(0)$.
\end{lem}

\begin{proof}
Fix $f\in\Sp$.
By Lemma~\ref{lem:T(1)=1},
$T(\mathbf1)=\mathbf1$.
Since $T$ is norm-additive,
\[
\Vp{\mathbf{1}+T(f)}
=
\Vp{T(\mathbf{1})+T(f)}
=\Vp{\mathbf{1}+f}.
\]
Applying \eqref{lem2.2.1} twice,
\[
\alpha(T(f)(0))
=\Vp{\mathbf{1}+T(f)}^p
=\Vp{\mathbf{1}+f}^p
=\alpha(f(0)).
\]
Since $\alpha$ is injective,
we conclude that $T(f)(0)=f(0)$.
\end{proof}

\subsection{The induced map
and proof of the main theorem}

Having established the preservation of the initial value under $T$, we can
now transfer $T$ to $\B$ via the derivative map $D_p$.

Using the bijection $D_p$, we define the composite map
$\widetilde{T}\colon\B\to\B$ by
\[
\widetilde{T}:=D_p\circ T\circ D_p^{-1}.
\]
To apply Theorem~\ref{Thm:auxiliary representation theorem ball},
we now verify that the induced map satisfies its hypotheses.
\[
\begin{CD}
\Sp@>{T}>>\Sp\\
@A{D_p^{-1}}AA
@VV{D_p}V\\
\B@>>{\widetilde{T}}>\B
\end{CD}
\]

\begin{prop}\label{prop:Ttilde-norm-additive}
$\widetilde{T}\colon\B\to\B$
is a surjective norm-additive map.
\end{prop}

\begin{proof}
Since $D_p$ is bijective and $T$ is surjective,
it follows that $\widetilde{T}=D_p\circ T\circ D_p^{-1}$ is surjective.

We prove that $\widetilde{T}$ is norm-additive.
Let $u,v\in\B$ and put
\[
f:=D_p^{-1}(u),\qquad
g:=D_p^{-1}(v).
\]
Since $T\colon\Sp\to\Sp$ is norm-additive,
$\Vp{T(f)+T(g)}^p=\Vp{f+g}^p$.
That is,
\[
\bigl(T(f)(0)+T(g)(0)\bigr)^p
+\VI{T(f)'+T(g)'}^p
=\bigl(f(0)+g(0)\bigr)^p
+\VI{f'+g'}^p.
\]
By Lemma~\ref{lem:T(f)(0)=f(0)},
$T(h)(0)=h(0)$ for every $h\in\Sp$. The
preceding equality yields
\[
\VI{T(f)'+T(g)'}
=\VI{f'+g'}.
\]
Combining $\widetilde{T}=D_p\circ T\circ D_p^{-1}$ with
the preceding equalities shows that
\[
\widetilde{T}(u)=D_p(T(f))=T(f)',\qquad
\widetilde{T}(v)=D_p(T(g))=T(g)'.
\]
By the choice of $f,g$,
we obtain $f'=D_p(f)=u$ and $g'=D_p(g)=v$.
Combining the preceding equalities gives
\[
\VI{\widetilde{T}(u)+\widetilde{T}(v)}
=\VI{T(f)'+T(g)'}
=\VI{f'+g'}
=\VI{u+v},
\]
which proves the norm-additivity of $\widetilde{T}$.
\end{proof}

We can now complete the proof of
Theorem~\ref{thm1.1} for $1<p<\infty$.

\begin{proof}[\textbf{Proof of Theorem~\ref{thm1.1} for $1<p<\infty$.}]
By Proposition~\ref{prop:Ttilde-norm-additive},
$\widetilde T\colon\B\to\B$ is a surjective norm-additive map.
Therefore, Theorem~\ref{Thm:auxiliary representation theorem ball},
applied with $X=Y=I$, yields a homeomorphism
$\tau\colon I\to I$ such that
\[
\widetilde T(u)=u\circ\tau
\qquad (u\in\B).
\]
Let $f\in\Sp$.
Since $\widetilde{T}=D_p\circ T\circ D_p^{-1}$,
we have $D_p\circ T=\widetilde{T}\circ D_p$.
Applying the preceding equality 
with $D_p(g)=g'$ for $g\in\Sp$ gives
\[
T(f)'=D_p(T(f))
=\widetilde{T}(D_p(f))
=f'\circ\tau.
\]
Hence $T(f)'=f'\circ\tau$.
Since $T(f)(0)=f(0)$ by Lemma~\ref{lem:T(f)(0)=f(0)},
the Fundamental Theorem of Calculus
gives
\[
T(f)(t)=f(0)+\int_0^tf'(\tau(s))\,ds
\]
for every $t\in I$.
This proves the forward implication
of Theorem~\ref{thm1.1} for $1<p<\infty$.
\end{proof}

\section{The case $p=\infty$}\label{section4}

We now prove the forward implication of
Theorem~\ref{thm1.1} for $p=\infty$.
Throughout this section, set
\begin{align*}
\Sinf:=\{f\in C^1(I):f(0)\geq0,\ f'\geq0,\ \Vinf{f}=1\},\\
\|f\|_{(\infty)}:=\max\{|f(0)|,\VI{f'}\}
\qquad(f\in C^1(I)).
\end{align*}
Let $T\colon\Sinf\to\Sinf$ be a surjective norm-additive map; that is,
\[
\|T(f)+T(g)\|_{(\infty)}
=
\|f+g\|_{(\infty)}
\qquad (f,g\in\Sinf).
\]
Our strategy is to identify $\Sinf$ with a positive unit sphere
of a suitable space of continuous functions and then apply
Theorem~\ref{Thm:auxiliary representation theorem}.

\subsection{Identification with a positive unit sphere}

\leavevmode
Let
$L:=\{-1\}\cup I$.
For each $f\in\Sinf$, define $\widehat{f}\in C(L)$ by
\begin{equation}\label{eq:hat-f}
\widehat{f}(x)=
\begin{cases}
f(0), & x=-1,\\
f'(x), & x\in I.
\end{cases}
\end{equation}
For every $f\in\Sinf$, we have $f(0)\geq0$, $f'\geq0$, and
$\VL{\widehat{f}}=\max\{f(0),\V{f'}_{I}\}
=\Vinf{f}=1$.
Hence $\widehat{f}\in\SCL$, and thus the map
$D_\infty\colon\Sinf\to\SCL$, defined by
\[
D_\infty(f)=\widehat{f}
\qquad (f\in\Sinf)
\]
is well defined.
We next verify that $D_\infty$ is a bijective norm-additive map.

\begin{lem}\label{lem:D_infty}
The map $D_\infty$ is a bijective
norm-additive map with inverse
\[
D_\infty^{-1}(u)(t)=u(-1)+\int_0^t u(s)\,ds 
\]
for all $u\in\SCL$ and $t\in I$.
\end{lem}

\begin{proof}
Define
\[
\Jinf(u)(t):=u(-1)+\int_0^tu(s)\,ds
\qquad(u\in\SCL,\ t\in I).
\]
Fix an arbitrary $u\in\SCL$. We prove that $\Jinf(u)\in\Sinf$.
Since $u\in\SCL$, we have $u\geq0$ on $L$ and $\V{u}_L=1$. 
Hence
\[
\Jinf(u)(0)=u(-1)\geq0,\qquad
\Jinf(u)'=u\geq0
\quad\mbox{on  $I$}.
\]
Moreover, we have
\[
\Vinf{\Jinf(u)}
=\max\{\Jinf(u)(0),\VI{\Jinf(u)'}\}
=\max\{u(-1),\VI{u}\}
=\V{u}_L=1.
\]
These show that $\Jinf(u)\in\Sinf$.
Thus $\Jinf\colon\SCL\to\Sinf$ is well defined.

We prove that $D_\infty$ is a bijective map
with its inverse $\Jinf$.
By the
definition of $D_\infty$ and the preceding equalities,
we obtain
\[
D_\infty(\Jinf(u))(x)
=\widehat{\Jinf(u)}(x)=
\begin{cases}
\Jinf(u)(0),&x=-1,\\
\Jinf(u)'(x),&x\in I
\end{cases}
=u(x)
\]
for every $u\in\SCL$ and $x\in L$.
Thus $D_\infty\circ\Jinf=\operatorname{id}_{\SCL}$.
Next, we show that $\Jinf\circ D_\infty=\operatorname{id}_{\Sinf}$. Fix an arbitrary $f\in\Sinf$.
By definition,
$D_\infty(f)=\widehat{f}$.
Applying \eqref{eq:hat-f},
$\widehat{f}(-1)=f(0)$ and $\widehat{f}=f'$ on $I$.
By the defining property of $\Jinf$, we have
\[
\Jinf(D_\infty(f))(t)
=\widehat{f}(-1)+\int_0^t\widehat{f}(s)\,ds
=f(0)+\int_0^tf'(s)\,ds=f(t)
\]
for every $t\in I$.
Thus 
$\Jinf\circ D_\infty=\operatorname{id}_{\Sinf}$.
Therefore we have proved that $D_\infty$ is a bijective map
with its inverse $\Jinf$.

Finally, we prove that $D_\infty$ is norm-additive. 
Let $f,g\in \Sinf$.
By the definition of $D_\infty$, we have
$D_\infty(h)(-1)=\widehat{h}(-1)=h(0)$
and $D_\infty(h)=h'$ on $I$
for $h\in \Sinf$.
This implies that
\[
\V{D_\infty(f)+D_\infty(g)}_L
=\max\{f(0)+g(0),\VI{f'+g'}\}
=\Vinf{f+g}.
\]
Hence,
$D_\infty$ is norm-additive.
\end{proof}

\subsection{Application of the sphere representation theorem}

Having identified $\Sinf$ with $\SCL$ via $D_\infty$,
we now transfer $T$ to $\SCL$.
Define
$\Delta\colon\SCL\to\SCL$ by
\[
\Delta:=D_\infty\circ T\circ D_\infty^{-1}.
\]
Since $D_\infty$ is a bijective norm-additive map
by Lemma~\ref{lem:D_infty},
its inverse $D_\infty^{-1}$ is also norm-additive.
Together with the surjectivity and norm-additivity of $T$,
this shows that $\Delta$ is a surjective norm-additive map.
\[
\begin{CD}
\Sinf@>{T}>>\Sinf\\
@A{D_\infty^{-1}}AA
@VV{D_\infty}V\\
\SCL@>>{\Delta}>\SCL
\end{CD}
\]

We are now ready to prove the forward implication of the main theorem
for $p=\infty$.

\begin{proof}[\textbf{Proof of Theorem~\ref{thm1.1} for $p=\infty$.}]
Since $\Delta\colon\SCL\to\SCL$
is a surjective norm-additive map, Theorem~\ref{Thm:auxiliary representation theorem}, applied with
$X=Y=L$, yields a homeomorphism
$\widetilde\tau\colon L\to L$ satisfying
\begin{equation*}
\Delta(u)=u\circ\widetilde\tau
\qquad (u\in\SCL).
\end{equation*}
Since $-1$ is the unique isolated point of $L$, we have
$\widetilde\tau(-1)=-1$.
Set $\tau:=\widetilde{\tau}|_{I}$.
Then $\tau\colon I\to I$ is a homeomorphism.
Now let $f\in \Sinf$, and fix $t\in I$. Since
$\Delta=D_\infty\circ T\circ D_\infty^{-1}$, we have
$D_\infty\circ T=\Delta\circ D_\infty$.
Since $D_\infty(f)=\widehat{f}$, the representation
formula of $\Delta$ shows that
\[
\Delta(D_\infty(f))=\widehat{f}\circ\widetilde{\tau}.
\]
Combining this with the preceding equality gives
$T(f)=D_\infty^{-1}(\widehat{f}\circ\widetilde{\tau})$.
Applying the formula for $D_\infty^{-1}$ in Lemma~\ref{lem:D_infty}, we obtain
\[
T(f)(t)
=
(\widehat f\circ\widetilde\tau)(-1)
+\int_0^t(\widehat f\circ\widetilde\tau)(s)\,ds.
\]
Since $\widetilde\tau(-1)=-1$, we have
$(\widehat f\circ\widetilde\tau)(-1)=\widehat f(-1)=f(0)$.
Moreover, for every $s\in I$, we have
$\widetilde\tau(s)=\tau(s)$, and hence
$(\widehat f\circ\widetilde\tau)(s)=\widehat f(\tau(s))
=f'(\tau(s))$.
Therefore,
\[
T(f)(t)
=
f(0)+\int_0^t f'(\tau(s))\,ds
\qquad(t\in I).
\]
This proves the forward implication of
Theorem~\ref{thm1.1} when $p=\infty$.
\end{proof}

\section{The converse implication and
the linear extension}\label{section5}

In Sections~\ref{section3} and~\ref{section4}, we proved the forward implication of
Theorem~\ref{thm1.1} for $1<p<\infty$ and $p=\infty$,
respectively, and obtained the representation of every surjective
norm-additive map in terms of a homeomorphism of $I$.
We now prove the converse implication and the linear extension result.
In fact, the maps arising from such homeomorphisms have the required
properties for all $1\leq p\leq\infty$.

Let $\tau\colon I\to I$ be a homeomorphism,
and define, for each $1\leq p\leq\infty$,
$U_\tau\colon\Sp\to C^1(I)$ by
\[
U_\tau(f)(t)
=f(0)+\int_0^t f'(\tau(s))\,ds
\qquad(f\in\Sp,\,t\in I).
\]

We now verify that the map $U_\tau$ has the properties needed for
the converse implication and the linear extension result.

\begin{prop}\label{prop:converse implication}
Let $1\le p\le\infty$.
Then $U_\tau$ is a surjective norm-additive map
from $\Sp$ onto itself.
Moreover, $U_\tau$ extends to a surjective linear isometry
from $(C^1(I),\Vp{\cdot})$ onto itself.
\end{prop}

The first assertion of Proposition~\ref{prop:converse implication}
proves the converse implication of Theorem~\ref{thm1.1}.

\begin{proof}
Let $1\leq p\leq\infty$.
Define a map $\widetilde{U}_{\tau}\colon C^1(I)\to C^1(I)$ by
\[
\widetilde U_\tau(f)(t)=f(0)+\int_0^tf'(\tau(s))\,ds
\qquad(f\in C^1(I),\,t\in I).
\]
By definition, $\widetilde U_\tau$ is a linear map
from $C^1(I)$ to itself,
and $\widetilde U_\tau=U_\tau$ on $\Sp$.
We show that $\widetilde U_\tau$ is surjective.
Let $f\in C^1(I)$.
Then
\[
\widetilde{U}_{\tau^{-1}}(f)(t)
=f(0)+\int_0^tf'(\tau^{-1}(s))\,ds
\qquad(t\in I).
\]
This implies that
\[
\widetilde{U}_{\tau^{-1}}(f)(0)=f(0),\qquad
\widetilde{U}_{\tau^{-1}}(f)'=f'\circ\tau^{-1}
\quad\mbox{on $I$.}
\]
Combining this with the defining property
of $\widetilde{U}_{\tau}$ shows that
\[
(\widetilde{U}_{\tau}\circ
\widetilde{U}_{\tau^{-1}})(f)(t)
=f(0)+\int_0^t(f'\circ\tau^{-1})(\tau(s))\,ds
=f(t)
\]
for every $t\in I$.
Therefore $\widetilde{U}_\tau\circ\widetilde{U}_{\tau^{-1}}$ is
the identity map.
This proves that $\widetilde{U}_{\tau}$ is surjective.
Since $\widetilde{U}_{\tau}(f)(0)=f(0)$
and $\widetilde{U}_{\tau}(f)'=f'\circ\tau$ on $I$,
if $1\leq p<\infty$, then
\[
\Vp{\widetilde{U}_{\tau}(f)}^p
=|\widetilde{U}_{\tau}(f)(0)|^p+\VI{\widetilde{U}_{\tau}(f)'}^p
=|f(0)|^p+\VI{f'\circ\tau}^p
=\Vp{f}^p,
\]
where we have used that $\tau$ is surjective.
In the same way, if $p=\infty$, then
\[
\Vinf{\widetilde{U}_{\tau}(f)}
=\max\{|\widetilde{U}_{\tau}(f)(0)|,\VI{\widetilde{U}_{\tau}(f)'}\}
=\max\{|f(0)|,\VI{f'\circ\tau}\}
=\Vinf{f}.
\]
Since $f\in C^1(I)$ was arbitrary,
this shows that $\widetilde{U}_{\tau}$ is norm-preserving.
Since $\widetilde{U}_{\tau}$ is linear and norm-preserving,
it is an isometry.
Therefore, $\widetilde{U}_{\tau}$ is a surjective linear isometry
from $(C^1(I),\Vp{\cdot})$ onto itself that extends $U_{\tau}$.

Next we prove that $U_\tau$ is a surjective map
from $\Sp$ onto itself.
Let $f\in\Sp$. Then
\[
\widetilde{U}_{\tau}(f)(0)=f(0)\geq0,\qquad
\widetilde{U}_{\tau}(f)'=f'\circ\tau\geq0
\quad\mbox{on $I$}.
\]
Since $\widetilde{U}_{\tau}$ is norm-preserving,
we have $\Vp{\widetilde{U}_\tau(f)}=\Vp{f}=1$.
Since $f\in\Sp$ was arbitrary, we conclude that
$\widetilde{U}_\tau(\Sp)\subset\Sp$.
Applying the same argument as above,
we obtain $\widetilde{U}_{\tau^{-1}}(\Sp)\subset\Sp$.
Since $\widetilde{U}_\tau\circ\widetilde{U}_{\tau^{-1}}$ is
the identity map, we obtain
$\Sp\subset\widetilde{U}_\tau(\Sp)$.
Therefore $\widetilde{U}_\tau(\Sp)=\Sp$.
Since $\widetilde{U}_\tau$ is an extension
of $U_\tau$, we conclude $U_\tau(\Sp)=\Sp$.
Thus $U_\tau$ is a surjective map
from $\Sp$ onto itself.

Finally, we prove that $U_\tau$ is norm-additive.
Let $f,g\in\Sp$.
Since $\widetilde{U}_\tau$ is a linear extension of $U_\tau$,
which preserves the norm $\Vp{\cdot}$, we have
\[
\Vp{U_\tau(f)+U_\tau(g)}
=\Vp{\widetilde U_\tau(f)+\widetilde U_\tau(g)}
=\Vp{\widetilde U_\tau(f+g)}
=\Vp{f+g}.
\]
Therefore $U_\tau$ is a surjective norm-additive
map from $\Sp$ onto itself.
\end{proof}

Having completed the proof of Theorem~\ref{thm1.1},
we now prove Corollary~\ref{cor1.2}.

\begin{proof}[\textbf{Proof of Corollary~\ref{cor1.2}}]
For both
 $1<p<\infty$ and $p=\infty$, Theorem~\ref{thm1.1}
provides a homeomorphism $\tau\colon I\to I$ satisfying
\[
T(f)(t)=f(0)+\int_0^t f'(\tau(s))\,ds
\qquad (f\in\Sp,\ t\in I).
\]
It is immediate from Proposition~\ref{prop:converse implication}
that $T$ extends to a surjective linear isometry
from $(C^1(I),\Vp{\cdot})$ onto itself.
\end{proof}

\section{Failure of rigidity for $p=1$}\label{section6}

Proposition~\ref{prop:converse implication} shows that
the converse implication of Theorem~\ref{thm1.1} remains
valid for $p=1$. In contrast, the forward implication of
Theorem~\ref{thm1.1} fails at the endpoint $p=1$.
Moreover, the linear extension conclusion of Corollary~\ref{cor1.2}
also fails. We give a single counterexample exhibiting both failures.

\begin{counterexample}
When $p=1$, there exists a surjective norm-additive map
\[
T\colon\mathcal S_1^+\to\mathcal S_1^+
\]
which is not of the form
\[
T(f)(t)=f(0)+\int_0^t f'(\tau(s))\,ds
\]
for any homeomorphism $\tau\colon I\to I$,
and which admits no linear extension to $C^1(I)$.

For $r\in I$, define
\[
f_r(t):=r+(1-r)t
\qquad (t\in I).
\]
Then 
$f_r(0)=r$ and $f_r'=(1-r)\mathbf{1}$.
Thus $\V{f_r}_{(1)}=1$ and hence
$f_r\in \mathcal{S}_1^+$.
Choose a bijection $\phi\colon I\to I$ such that
$\phi(1)\neq1$.
Define $T_\phi\colon \mathcal{S}_1^+\to\mathcal{S}_1^+$ by
\[
T_\phi(f)=
\begin{cases}
f_{\phi(r)}, & \text{if } f=f_r \text{ for some } r\in I,\\
f, & \text{otherwise},
\end{cases}
\qquad (f\in\mathcal{S}_1^+).
\]
We claim that $T_\phi$ has all the properties stated above.
Since $\phi$ is bijective, we have
$T_\phi^{-1}=T_{\phi^{-1}}$,
and hence $T_\phi$ is surjective.

We next prove that $T_\phi$ is norm-additive.
Let $f,g\in\mathcal{S}_1^+$.
We distinguish the following two cases.

\medskip
\noindent
{\bf Case 1.}
At least one of $f$ and $g$ is affine.

By symmetry, we may assume that $f=f_r$ for some $r\in I$.
Then
$f(0)=r$
and
$f'=(1-r)\mathbf{1}$.
Since $g'\ge0$ and $(1-r)\mathbf{1}$ is a nonnegative constant function, we have
\[
\V{(1-r)\mathbf{1}+g'}_{I}
=
(1-r)+\V{g'}_{I}.
\]
Therefore,
\[
\begin{aligned}
\|f+g\|_{(1)}
&=
r+g(0)+\V{(1-r)\mathbf{1}+g'}_{I}\\
&=
1+g(0)+\V{g'}_{I}=2,
\end{aligned}
\]
where the last equality follows from
$g\in\mathcal{S}_1^+$.

On the other hand,
$T_\phi(f)=f_{\phi(r)}$
and hence
$T_\phi(f)(0)=\phi(r)$
and
$T_\phi(f)'=(1-\phi(r))\mathbf{1}$.
Since $T_\phi(g)'\ge0$ and
$(1-\phi(r))\mathbf{1}$ is a nonnegative constant function, we have
\[
\V{(1-\phi(r))\mathbf{1}+T_\phi(g)'}_{I}
=
1-\phi(r)+\V{T_\phi(g)'}_{I}.
\]
Therefore,
\[
\begin{aligned}
\|T_\phi(f)+T_\phi(g)\|_{(1)}
&=
\phi(r)+T_\phi(g)(0)+\|(1-\phi(r))\mathbf{1}+T_\phi(g)'\|_{I}\\
&=
1+T_\phi(g)(0)+\|T_\phi(g)'\|_I
=2,
\end{aligned}
\]
where the last equality follows
from $T_\phi(g)\in\mathcal{S}_1^+$.
Thus
\[
\|T_\phi(f)+T_\phi(g)\|_{(1)}
=
\|f+g\|_{(1)}.
\]

\medskip
\noindent
{\bf Case 2.}
Neither $f$ nor $g$ is affine.

In this case,
$T_\phi(f)=f$
and $T_\phi(g)=g$.
Therefore,
$\|T_\phi(f)+T_\phi(g)\|_{(1)}
=\|f+g\|_{(1)}$.
Since $f$ and $g$ were arbitrary, we conclude that $T_\phi$ is
norm-additive.

We next show that $T_\phi$ does not have the representation in Theorem~\ref{thm1.1}. Since $\mathbf1=f_1$, we have
\[
T_\phi(\mathbf1)=f_{\phi(1)}.
\]
Since $\phi(1)\neq1$, we have
$f_{\phi(1)}\neq\mathbf1$, and hence
\[
T_\phi(\mathbf1)\neq\mathbf1.
\]
On the other hand, every map of the form
\[
T(f)(t)=f(0)+\int_0^t f'(\tau(s))\,ds
\]
satisfies $T(\mathbf1)=\mathbf1$.
Therefore $T_\phi$ does not have the representation
in Theorem~\ref{thm1.1}. Thus the forward implication
of Theorem~\ref{thm1.1} fails for $p=1$.

Finally, we prove that $T_\phi$ admits no linear extension
to $C^1(I)$.
Suppose, to the contrary, that there exists a linear map
$\Lambda\colon C^{1}(I)\to C^{1}(I)$
such that $\Lambda(f)=T_\phi(f)$ for every
$f\in\mathcal{S}_1^+$.
Define
\[
g(t)=\frac12+\frac{t^2}{4}
\qquad (t\in I).
\]
Then $g(0)=1/2$, $g'(t)=t/2$ for $t\in I$, and
$\V{g'}_{I}=1/2$.
Hence 
$\V{g}_{(1)}=1$ and $g\in\mathcal{S}_1^+$.
Set
\[
h
=\frac12g+\frac12\mathbf1.
\]
Then
$h(t)=3/4+t^2/8$,
$h(0)=3/4$,
$h'(t)=t/4$,
and
$\|h'\|_{I}=1/4$.
Hence $h\in\mathcal{S}_1^+$.
Since $h$ is not affine, $T_\phi(h)=h$.
Since $\Lambda$ is a linear extension of $T_\phi$,
\[
h=\Lambda(h)
=
\Lambda\!\left(\frac12g+\frac12\mathbf1\right)
=
\frac{1}{2}\Lambda(g)+\frac{1}{2}\Lambda(\mathbf1)
=
\frac{1}{2}T_\phi(g)+\frac{1}{2}T_\phi(\mathbf1).
\]
Since $g$ is not affine, we have $T_\phi(g)=g$.
Moreover, since $\mathbf1=f_1$ is affine,
it follows from the definition of $T_\phi$ that
$T_\phi(\mathbf1)=f_{\phi(1)}$. Thus
\[
h=\frac{1}{2}g+\frac{1}{2}f_{\phi(1)}.
\]
On the other hand, by the definition of $h$,
we have
$g+f_{\phi(1)}=2h=g+\mathbf{1}$.
Thus $f_{\phi(1)}=\mathbf{1}$.
This shows that
$\phi(1)+(1-\phi(1))t=f_{\phi(1)}(t)=1$
for all $t\in I$.
Taking $t=0$, we obtain $\phi(1)=1$, which contradicts
the assumption that $\phi(1)\neq1$.
Hence $T_\phi$ admits no linear extension to
$C^{1}(I)$.
\end{counterexample}

\section*{Statements and Declarations}

\noindent\textbf{Funding.}
The second author was partially supported by the Japan Society for the
Promotion of Science (JSPS) KAKENHI Grant Number JP 25K07028.
\medskip

\noindent\textbf{Competing Interests.}
The authors have no relevant financial or non-financial interests
to disclose.

\medskip

\noindent\textbf{Data Availability.}
Data sharing is not applicable to this article as no datasets were
generated or analyzed during the current study.

\medskip

\noindent\textbf{Use of generative AI.}
During the preparation of this manuscript, the authors used ChatGPT
(OpenAI) for language editing, literature checking, and assistance in
reviewing the exposition and mathematical arguments. All mathematical
content, references, and the final text were independently verified and
approved by the authors, who take full responsibility for the manuscript.

\begin{thebibliography}{99}

\bibitem{CDZ}
L.~Chen, Y.~Dong and B.~Zheng,
\newblock On norm-additive maps between the maximal groups of positive continuous functions,
\newblock {\em Results Math.} \textbf{74} (2019), no.~4, Art.~152.
%\newblock \url{https://doi.org/10.1007/s00025-019-1076-x}

\bibitem{CHMP}
M.~Cueto-Avellaneda, D.~Hirota, T.~Miura and A.M.~Peralta,
\newblock Exploring new solutions to Tingley's problem for function algebras,
\newblock {\em Quaest. Math.} \textbf{46} (2023), no.~7, 1315--1346.
%\newblock \url{https://doi.org/10.2989/16073606.2022.2072787}

\bibitem{CP}
M.~Cueto-Avellaneda and A.M.~Peralta,
\newblock On the Mazur--Ulam property for the space of Hilbert-space-valued continuous functions,
\newblock {\em J. Math. Anal. Appl.} \textbf{479} (2019), no.~1, 875--902.
%\newblock \url{https://doi.org/10.1016/j.jmaa.2019.06.056}

\bibitem{enam}
Y.~Enami, D.~Hirota, H.~Koshimizu and M.-R.~Lin,
\newblock Surjective isometries on the positive parts of the unit spheres of some function spaces,
\newblock arXiv:2606.30952, 2026.
%\newblock \url{https://arxiv.org/abs/2606.30952}

\bibitem{HDL}
J.~Hao, Y.~Dong and L.~Li,
\newblock Maps preserving the norm of the positive sum in $L^p$ spaces,
\newblock {\em Acta Math. Sci. Ser. B (Engl. Ed.)}
\textbf{42} (2022), no.~2, 789--794.
%\newblock \url{https://doi.org/10.1007/s10473-022-0223-8}


\bibitem{Hatori2022}
O.~Hatori,
\newblock The Mazur--Ulam property for uniform algebras,
\newblock {\em Studia Math.} \textbf{265} (2022), no.~2, 227--239.
%\newblock \url{https://doi.org/10.4064/sm210703-11-9}

\bibitem{Hatori2023}
O.~Hatori,
\newblock The Mazur--Ulam property for a Banach space which satisfies a separation condition,
\newblock {\em RIMS K\^oky\^uroku Bessatsu} \textbf{B93} (2023), 29--82.

\bibitem{HOS}
O.~Hatori, S.~Oi and R.~Shindo~Togashi,
\newblock Tingley's problems on uniform algebras,
\newblock {\em J. Math. Anal. Appl.} \textbf{503} (2021), no.~2, Paper No.~125346.
%\newblock \url{https://doi.org/10.1016/j.jmaa.2021.125346}

\bibitem{HM}
D.~Hirota and T.~Miura,
\newblock Tingley's problem for a Banach space of Lipschitz functions on the closed unit interval,
\newblock {\em RIMS K\^oky\^uroku Bessatsu} \textbf{B93} (2023), 157--181.

\bibitem{LNW}
C.-W.~Leung, C.-K.~Ng and N.-C.~Wong,
\newblock On a variant of Tingley's problem for some function spaces,
\newblock {\em J. Math. Anal. Appl.} \textbf{496} (2021), no.~1, Paper No.~124800.
%\newblock \url{https://doi.org/10.1016/j.jmaa.2020.124800}

\bibitem{peralta}
A.M.~Peralta,
\newblock A survey on Tingley's problem for operator algebras,
\newblock {\em Acta Sci. Math. (Szeged)} \textbf{84} (2018), no.~1--2, 81--123.
%\newblock \url{https://doi.org/10.14232/actasm-018-255-0}

\bibitem{SS}
L.~Sun and Y.~Sun,
\newblock On extension of norm-additive maps between the positive unit spheres of $\ell_q(\ell_p)$,
\newblock {\em Results Math.} \textbf{79} (2024), no.~3, Art.~126.
%\newblock \url{https://doi.org/10.1007/s00025-024-02154-y}

\bibitem{ting}
D.~Tingley,
\newblock Isometries of the unit sphere,
\newblock {\em Geom. Dedicata} \textbf{22} (1987), no.~3, 371--378.
%\newblock \url{https://doi.org/10.1007/BF00147942}

\bibitem{Wang1994}
R.~Wang,
\newblock Isometries between the unit spheres of $C_0(\Omega)$ type spaces,
\newblock {\em Acta Math. Sci. Ser. B (Engl. Ed.)} \textbf{14} (1994), no.~1, 82--89.
%\newblock \url{https://doi.org/10.1016/S0252-9602(18)30093-6}

\bibitem{Wang1996}
R.~Wang,
\newblock Isometries of $C_0^{(n)}(X)$,
\newblock {\em Hokkaido Math. J.} \textbf{25} (1996), no.~3, 465--519.
%\newblock \url{https://doi.org/10.14492/hokmj/1351516747}

\bibitem{ZTW_lattice}
J.~Zhang, M.-C.~Tsai and N.-C.~Wong,
\newblock Norm of positive sum preservers of smooth Banach lattices and $L^p(\mu)$ spaces,
\newblock {\em J. Nonlinear Convex Anal.} \textbf{20} (2019), no.~12, 2613--2621.

\bibitem{ZTW_noncommutative}
J.~Zhang, M.-C.~Tsai and N.-C.~Wong,
\newblock Norm of positive sum preservers of noncommutative $L^p(M)$ spaces,
\newblock {\em J. Nonlinear Convex Anal.} \textbf{22} (2021), no.~2, 251--264.

\end{thebibliography}
\end{document}